\documentclass[a4paper,12pt]{article}
\usepackage{a4wide}
\usepackage{amsmath}
\usepackage{amssymb}
\usepackage{amsthm}
\usepackage{latexsym}
\usepackage{graphicx}
\usepackage[english]{babel}
\usepackage{makeidx}
\usepackage{mathtools}

\newtheorem{teor}[subsection]{Theorem}

\newtheorem{cor} [subsection]{Corollary}

\newcommand{\Res}{Res}

\def\Gal{\operatorname{Gal}}

\def\Res{\operatorname{Res}}
\def\Ind{\operatorname{Ind}}

\def\Ker{\operatorname{Ker}}
\def\SL{\operatorname{SL}}
\def\GL{\operatorname{GL}}
\numberwithin{equation}{section}

\begin{document}
\selectlanguage{english}
\frenchspacing

\large
\begin{center}
\textbf{Artin L-functions to almost monomial Galois groups}

Mircea Cimpoea\c s and Florin Nicolae
\end{center}
\normalsize

\begin{abstract}

If $K/\mathbb Q$ is a finite Galois extension with an almost monomial Galois group and if 
 $s_0\in\mathbb C\setminus\{1\}$ is not a common zero for any two Artin L-functions associated to distinct complex irreducible characters of the 
Galois group then all Artin L-functions of $K/\mathbb Q$ are holomorphic at  $s_0$. We present examples and basic properties of almost monomial groups.


\end{abstract}

\section*{Introduction}

A finite group $G$ is called {\em almost monomial} if for every 
distinct complex irreducible characters $\chi$ and $\psi$ of  $G$ there exist a subgroup $H$ of $G$ and a linear character $\varphi$ of $H$ such 
that the induced character $\Ind_H^G\varphi$ contains  $\chi$ and does not contain  $\psi$. This definition appears in \cite{monat} in connection with the 
study of the holomorphy of Artin L-functions associated to a finite Galois extension of $\mathbb Q$ at a point in the complex plane. Let $K/\mathbb Q$ 
be a finite Galois extension with the Galois group 
$G$. Let $\chi_1,\ldots,\chi_r$ be the complex irreducible characters of $G$, $f_1=L(s, \chi_1),\ldots,f_r=L(s,\chi_r)$ 
the corresponding Artin L-functions. Let $s_0\in\mathbb C\setminus\{1\}$. Our main result is Theorem 1.1: if 
$G$ is almost monomial and $s_0$ is not a common zero for any two distinct L-functions $f_k$ and $f_l$ then 
all Artin L-functions of $K/\mathbb Q$ are holomorphic at $s_0$. In the second 
section of this paper we give examples and prove some basic properties of almost monomial groups. Every monomial group is almost monomial. The special 
linear group $\SL_2(\mathbb F_3)$ is almost monomial, but the general linear group $\GL_2(\mathbb F_3)$ is not. The alternating group  $A_n$ is almost 
monomial for $n=1,2,3,4,5,9$ and is not for $n=6,7,10,11,12$. In Theorem 2.1 we prove that the symmetric group is almost monomial for any $n\geq 1$. 
A subgroup of an almost monomial group is not necessarily almost monomial. In Theorem 2.2 we prove that a factor group of an almost monomial group is 
almost monomial. A finite product of finite groups is 
almost monomial if and only if each of them is almost monomial.  

\section{ The main result }

Let $K/\mathbb Q$ be a finite Galois extension. 
For the character $\chi$ of a representation of the Galois group $G:=\Gal(K/\mathbb Q)$
on a finite dimensional complex vector space let $L(s,\chi):=L(s,\chi,K/\mathbb Q)$ be the corresponding Artin L-function 
(\cite[P.296]{artin2}). 
Artin conjectured that $L(s,\chi)$ is holomorphic in $\mathbb C\setminus \{1\}$.   
Brauer proved that $L(s,\chi)$ is meromorphic in $\mathbb C$.
Let $\chi_1,\ldots,\chi_r$ be the irreducible characters of $G$, $f_1=L(s, \chi_1),\ldots,f_r=L(s,\chi_r)$ the corresponding Artin L-functions,
$$Ar:=\{f_1^{k_1}\cdot\ldots\cdot f_r^{k_r}\mid k_1\geq 0,\ldots,k_r\geq 0\}$$
the multiplicative semigroup of all  L-functions. For $s_0\in\mathbb C,s_0\neq 1$ let $ {\it Hol}(s_0)$ be the
subsemigroup of $Ar$ consisting of the L-functions which are holomorphic at
$s_0$. Artin's conjecture is:
$$ {\it Hol}(s_0)=Ar.$$

In (\cite{monat}) it was proved the following 

\noindent {\bf Theorem.} {\it 
If $G=\Gal(K/\mathbb Q)$ is almost monomial, then the following assertions are equivalent:

1) Artin's conjecture is true: $ {\it Hol}(s_0)=Ar.$

2) The semigroup $ {\it Hol}(s_0)$ is factorial.}

\noindent Our main result is

\begin{teor}
If $G$ is almost monomial and $s_0$ is not a common zero for any two distinct L-functions $f_k$ and $f_l$ then all Artin L-functions of $K/\mathbb Q$ 
are holomorphic at $s_0$.  
\end{teor}

\begin{proof}
Suppose that $s_0$ is a pole of some L-function $f_m$.   The
Dedekind zeta function of the field $K$ has the decomposition 
$$\zeta_{K}(s):=f_1^{d_1}\cdots f_r^{d_r},$$ where $d_j=\chi_j(1)$ for $1\leq j\leq r$. Since
$\zeta_{K}(s)$ is holomorphic at $s_0$ there exists $k\neq m$ such that 
$$f_k(s_0)=0.$$
Since $G$ is almost monomial there exist a subgroup $H$ of $G$ and a linear character $\lambda$ of $H$ such that 
$\Ind_H^G\lambda$ contains $\chi_m$ and does not contain  $\chi_k$. It holds that 
$$L(s, \Ind_H^G\lambda,K/\mathbb Q)=L(s, \lambda, K/F)$$
where $F$ is the fixed field of $H$. Since $\lambda$ is a linear character the L-function $L(s, \lambda, K/F)$ is a Hecke L-function hence is holomorphic at  $s_0$. It follows that $L(s, \Ind_H^G\lambda,K/\mathbb Q)$ is holomorphic at  $s_0$. Since $\Ind_H^G\lambda$ contains $\chi_m$ the 
 L-function $f_m$ is a factor of the L-function $L(s, \Ind_H^G\lambda,K/\mathbb Q)$. Since $f_m$ has a pole at $s_0$ and $L(s, \Ind_H^G\lambda,K/\mathbb Q)$ 
is holomorphic at $s_0$ there exist $l\neq m$ such that $f_l$ is a factor of $L(s, \Ind_H^G\lambda,K/\mathbb Q)$ and 
 $$f_l(s_0)=0.$$
Since $\Ind_H^G\lambda$ does not contain  $\chi_k$ we have that 
$$k\neq l,$$
so $s_0$ is a common zero for the distinct L-functions $f_k$ and $f_l$, which contradicts the hypothesis.
\end{proof}

\begin{cor}
 If $G$ is almost monomial and $s_0$ is a pole for a L-function $f_m$ then $s_0$ is a zero for two distinct L-functions 
$f_k$ and $f_l$.
\end{cor}

\section{Almost monomial groups}

In this section we present examples and prove some basic properties of almost monomial groups.

The following function in GAP \cite{gap} determines if a group $G$ is almost monomial:

\noindent
gap$>$ IsAlmostMonomial:=function(g)\\
$>$ cc:=ConjugacyClassesSubgroups(g);\\
$>$ i:=Size(Irr(g));\\
$>$ M:=IdentityMat(i);\\
$>$ for x in cc do\\
$>$ \;\; y:=Representative(x); \\
$>$ \;\; for z in Irr(y) do\\
$>$ \;\;\;\;   if z[1]=1 then o:=InducedClassFunction(z,g);\\
$>$ \;\;\;\;\;\;   l:=ConstituentsOfCharacter(o);\\
$>$ \;\;\;\;\;\;   for j in [1..i] do for k in [1..i] do \\
$>$ \;\;\;\;\;\;   if (Irr(g)[j] in l)and not(Irr(g)[k] in l) then M[j][k]:=1; fi;\\
$>$ \;\;\;\;   od;od;fi;\\
$>$ \;\; od;\\
$>$ od;\\
$>$ fals:=0;\\
$>$ for j in [1..i] do for k in [1..i] do \\
$>$ \;\; if M[j][k]=0 then fals:=1; fi; \\
$>$ od; od;\\
$>$ if fals=0 then return(true); else return(false); fi;\\
$>$ end;;\\

\noindent
The following code in GAP searches for groups with small number of elements which are not almost monomial:

\noindent
gap$>$ for s in [1..1000] do\\
$>$  for t in [1..NrSmallGroups(s)] do\\
$>$    if not(IsAlmostMonomial(SmallGroup(s,t))) then \\
\;\;\;\;\;\;\;Print("Small(",s,",",t,") is not AM $\backslash n$"); \\
$>$    else Print("Small(",s,",",t,") is AM $\backslash n$");\\
$>$ fi; od; od;\\

 \noindent The computations show that

\begin{itemize}
 \item The Mathieu groups $M_{11}$, $M_{12}$, $M_{22}$, $M_{23}$ and $M_{24}$ are not almost monomial.
 \item The Higman-Sims group $HS$ is not almost monomial. The Hall-Janko group $J_2$ is not almost monomial.
 \item The special linear group $\SL_2(\mathbb F_3)$ is almost monomial, but the general linear group $\GL_2(\mathbb F_3)$ is not.
       Note that $\SL_2(\mathbb F_3)$ is also the smallest solvable group which is not monomial.
 \item The groups $\SL_2(\mathbb F_{2^k})$ are almost monomial for $1\leq k\leq 5$.
 \item The groups $\SL_2(\mathbb F_5)$, $\SL_2(\mathbb F_7)$, $\SL_2(\mathbb F_{3^2})$, $\SL_2(\mathbb F_{11})$, $\SL_2(\mathbb F_{13})$, 
       $\SL_2(\mathbb F_{17})$, $\SL_2(\mathbb F_{19})$, $\SL_2(\mathbb F_{23})$,  $\SL_2(\mathbb F_{5^2})$, $\SL_2(\mathbb F_{3^3})$, 
       $\SL_2(\mathbb F_{29})$, $\SL_2(\mathbb F_{31})$, $\SL_2(\mathbb F_{37})$ are not almost monomial.
 \item The groups $\SL_3(\mathbb F_2)$, $\SL_3(\mathbb F_3)$, $\SL_3(\mathbb F_{2^2})$, $\SL_3(\mathbb F_5)$ are not almost monomial.
\end{itemize}

\noindent
The following code searches for alternating and symmetric groups which are/are not almost monomial:

\noindent
gap$>$ for s in [1..10] do\\
$>$  if IsAlmostMonomial(AlternatingGroup(s)) then Print("A",s," is AM");\\
$>$  else Print("A",s," is not AM");fi;Print("$\backslash n$");od;\\
gap$>$ for s in [1..10] do\\
$>$  if IsAlmostMonomial(SymmetricGroup(s)) then Print("S",s," is AM");\\
$>$  else Print("S",s," is not AM");fi;Print("$\backslash n$");od;\\

The computation shows that $A_n$ is almost monomial for $n=1,2,3,4,5,9$ and is not for $n=6,7,10,11,12$. The symmetric group $S_n$ is almost 
monomial for $n\leq 12$. 

We prove that $S_n$ is almost monomial for any $n\geq 1$. For a partition $\alpha$ of $n$ let $S_{\alpha}$ be the corresponding Young subgroup (\cite[Section 1.3, p. 16]{james}).   
Let $IS_{\alpha}$ be the trivial character of $S_{\alpha}$: $IS_{\alpha}(\sigma)=1$ for all $\sigma\in S_{\alpha}$. Let $AS_{\alpha
}$ be the alternating character of $S_{\alpha}$: $AS_{\alpha}(\sigma)=\text{\rm sgn}\,\sigma$ for all $\sigma\in S_{\alpha}$, where 
$\text{\rm sgn}\,\sigma$ is the signature of the permutation $\sigma$. Let $\alpha'$ be the associated partition with $\alpha$ in the sense of \cite[Section 1.4, p. 22]{james}. 
By \cite[Theorem 2.1.3, p. 35]{james} the characters $\Ind_{S_{\alpha}}^{S_n} IS_{\alpha}$ and $\Ind_{S_{\alpha'}}^{S_n} AS_{\alpha'}$ have exactly one irreducible constituent in common, denoted $[\alpha]$. By \cite[Theorem 2.1.11, p. 37]{james} any irreducible character of $S_n$
is of the form $[\alpha]$ for a suitable partition $\alpha$ of $n$. 

\begin{teor}
 The group $S_n$ is almost monomial for any $n\geq 1$.
\end{teor}

\begin{proof}
Let $\chi,\psi$ be two distinct irreducible characters of $S_n$. We have that $\chi=[\alpha]$ and $\psi=[\beta]$, where $\alpha,\beta$ are
two distinct partitions of $n$. If $\langle \Ind_{S_{\alpha}}^{S_n} IS_{\alpha},[\beta] \rangle = 0$ then by choosing $H:=S_{\alpha}$ and $\lambda:=IS_{\alpha}$
the definition of almost monomiality is satisfied. ( Here and in the sequel $\langle\,,\rangle$ is the usual scalar product of characters.) Assume that $\langle \Ind_{S_{\alpha}}^{S_n} IS_{\alpha},[\beta] \rangle \neq 0$. It follows that 
$\langle \Ind_{S_{\alpha'}}^{S_n} AS_{\alpha},[\beta] \rangle = 0$, otherwise $[\alpha]=[\beta]$, a contradiction. Choosing  $H:=S_{\alpha'}$ and $\lambda:=AS_{\alpha'}$
the definition of almost monomiality is again satisfied.
\end{proof}

If $H\leq G$ is a subgroup and $G$ is almost monomial then $H$ is not necessarily almost monomial, not even in the case when $H$ is normal.
For example if $G=S_6$ and $H=A_6 \unlhd G$ then $G$ is almost monomial but $H$ is not almost monomial. A factor group of an almost monomial group is 
almost monomial:

\begin{teor}
Let $N \unlhd G$ be a normal subgroup of the finite group $G$. If $G$ is almost monomial then $G/N$ is almost monomial.
\end{teor}

\begin{proof}
Let $\tilde{\chi}, \tilde{\psi}$ be two distinct irreducible characters of $G/N$ and let $\chi,\psi$ be their corresponding irreducible characters of $G$. 
We have that $N\subset \Ker(\chi)$ and $N\subset \Ker(\psi)$. Since $G$ is almost monomial there exist a subgroup $H\leqslant G$ and a linear character $\lambda$ of $H$
such that $\langle \Ind_H^G \lambda,\chi \rangle \neq 0$ and $\langle \Ind_H^G \lambda,\psi \rangle = 0$. 

\noindent We prove that $H\cap N\subset \Ker(\lambda)$. If $H\cap N\not\subset \Ker(\lambda)$ then the restricted character $\Res^H_{H\cap N} \lambda$
is linear and not trivial on $H\cap N$, so $$\langle \Res^H_{H\cap N} \lambda, 1_{H\cap N}   \rangle = 0,$$ where $1_{H\cap N}$ is the trivial 
character of $H\cap N$. Since $N\subset \Ker(\chi)$ it follows that $$\Res^G_{H\cap N} \chi = \chi(1)1_{H\cap N},$$
hence 
$$\langle \Res^H_{H\cap N} \lambda, \Res^G_{H\cap N} \chi  \rangle =\langle \Res^H_{H\cap N} \lambda, \chi(1)1_{H\cap N}\rangle=\chi(1)\langle \Res^H_{H\cap N} \lambda, 
1_{H\cap N}\rangle=0$$ 
and, by Frobenius reciprocity,
$$\langle \Ind_{H\cap N}^H(\Res^H_{H\cap N} \lambda), \Res^G_{H} \chi  \rangle = 0.$$ On the other hand
$$\langle \Ind_{H\cap N}^H(\Res^H_{H\cap N} \lambda), \lambda \rangle = \langle \Res^H_{H\cap N} \lambda, \Res^H_{H\cap N} \lambda \rangle=1,$$ hence
$$\langle \lambda, \Res^G_{H} \chi  \rangle = 0   $$ and, by Frobenius reciprocity,
$$\langle \Ind^G_H \lambda, \chi \rangle =0,$$ in contradiction with  $\langle \Ind_H^G \lambda,\chi \rangle \neq 0.$

\noindent We define the group homomorphism 
$$\tilde{\lambda}: H/(H\cap N) \cong HN/N \to \mathbb C^*,\; \tilde{\lambda}(hN):=\lambda(h), h\in H,$$
which is a linear character of the subgroup $HN/N$ of $G/N$. By Frobenius reciprocity
\begin{align*}
& \langle \Ind^{G/N}_{HN/N} \tilde{\lambda}, \tilde{\chi} \rangle = \langle \tilde{\lambda}, \Res^{G/N}_{HN/N} \tilde{\chi} \rangle = 
\frac{|H\cap N|}{|H|} \sum_{\bar h \in HN/N} \tilde{\lambda}(\bar h)\cdot \overline{\tilde{\chi}(\bar h)} = \\
& = \frac{1}{|H|} \sum_{h \in H}  \lambda( h)\cdot \overline{\chi( h)} = \langle \lambda, \Res^{G}_{H} \chi \rangle = \langle \Ind^{G}_{H} \lambda, \chi \rangle \neq 0.
\end{align*}
Similarly $$\langle \Ind^{G/N}_{HN/N} \tilde{\lambda}, \tilde{\psi} \rangle = \langle \Ind^{G}_{H} \lambda, \psi \rangle = 0,$$ so $G/N$ is almost monomial.
\end{proof}

We state without proof 
 
\begin{teor}
Let $G,G'$ be two finite groups. The following assertions are equivalent:
\begin{enumerate}
 \item[(1)] $G,G'$ are almost monomial.
 \item[(2)] $G\times G'$ is almost monomial.
\end{enumerate}
\end{teor}

{}

\vspace{2mm} \noindent {\footnotesize
\begin{minipage}[b]{15cm}
Mircea Cimpoea\c s, Simion Stoilow Institute of Mathematics of the Romanian Academy, Research unit 5, P.O.Box 1-764,\\
014700 Bucharest, Romania, E-mail: mircea.cimpoeas@imar.ro
\end{minipage}}

\vspace{2mm} \noindent {\footnotesize
\begin{minipage}[b]{15cm}
Florin Nicolae, Simion Stoilow Institute of Mathematics of the romanian Academy, P.O.Box 1-764,\\
014700 Bucharest, Romania, E-mail: florin.nicolae@imar.ro
\end{minipage}}
\end{document}